\documentclass[11pt]{article}
\usepackage[a4paper,margin=1in]{geometry}
\usepackage{float}
\usepackage[colorlinks,citecolor=magenta,linkcolor=black]{hyperref}
\pdfpagewidth=\paperwidth \pdfpageheight=\paperheight
\usepackage{amsfonts,amssymb,amsthm,amsmath,eucal,tabu,url}
\usepackage{pgf}
 \usepackage{array}
 \usepackage{tikz-cd}
 \usepackage{pstricks}
 \usepackage{pstricks-add}
 \usepackage{pgf,tikz}
 \usetikzlibrary{automata}
 \usetikzlibrary{arrows}
 \usepackage{indentfirst}
\usepackage{tabularx} 

\theoremstyle{plain}
\newtheorem{thm}{Theorem}[section]
\newtheorem{theorem}[thm]{Theorem}

\newtheorem{proposition}[thm]{Proposition}

\theoremstyle{definition}
\newtheorem{definition}[thm]{Definition}
\newtheorem{remark}[thm]{Remark}
\newtheorem{example}[thm]{Example}

\newtheorem{problem}[thm]{Problem}

\newtheorem{thevarthm}[thm]{\varthmname}

\newenvironment{varthm*}[1]{\trivlist\item[]{\bf #1.}\it}{\endtrivlist}

\def\keywordname{{\bfseries Keywords}}%
\def\keywords#1{\par\addvspace\medskipamount{\rightskip=0pt plus1cm
\def\and{\ifhmode\unskip\nobreak\fi\ $\cdot$
}\noindent\keywordname\enspace\ignorespaces#1\par}}
\def\subclassname{{\bfseries Mathematics Subject Classification
(2020)}\enspace}
\def\subclass#1{\par\addvspace\medskipamount{\rightskip=0pt plus1cm
\def\and{\ifhmode\unskip\nobreak\fi\ $\cdot$
}\noindent\subclassname\ignorespaces#1\par}}

\begin{document}
\title{On plus--one generated arrangements of plane conics}
\author{Artur Bromboszcz, Bartosz Jaros\l awski, Piotr Pokora}
\date{\today}
\maketitle
\thispagestyle{empty}
\begin{abstract}
In this paper, we examine the combinatorial properties of conic arrangements in the complex projective plane that possess certain quasi--homogeneous singularities. First, we introduce a new tool that enables us to characterize the property of being plus--one generated within the class of conic arrangements with some naturally chosen  quasi--homogeneous singularities. Next, we present a classification result on plus--one generated conic arrangements admitting only nodes and tacnodes as singularities. Building on results regarding conic arrangements with nodes and tacnodes, we present new examples of strong Ziegler pairs of conic-line arrangements -- that is, arrangements having the same strong combinatorics but distinct derivation modules.
\subclass{plus--one generated curves, conic-line arrangements, plane curve singularities, weak-combinatorics, strong combinatorics}
\end{abstract}
\section{Introduction}
In recent years, researchers have shown a growing interest in studying plus--one generated arrangements of curves. Recall that the plus--one generation is a homological property of the derivation module associated with arrangements of plane curves in the complex projective plane very close to being free. A reduced complex plane curve is free if the associated derivation module preserving the arrangement is a free module over the coordinate ring of the complex projective plane. The property of being plus--one generated relaxes the freeness condition. In the class of plus--one generated curves, new phenomena emerge that we want to understand. The concept of plus--one generatedness was introduced by Abe \cite{Abe18} in the setting of central hyperplane arrangements, and then it was generalized to any reduced plane curve by Dimca and Sticlaru \cite{DS20}. Let us notice here that plus--one generated hyperplane arrangements are important due to many factors, for instance they are sitting closely to the class of free arrangements via the addition/deletion technique. Here our main goal is to continue studies on plus--one generated arrangements of conics that were begun in papers by several authors, for instance in \cite{her, DJP, DS25, MP, MP1}. Let us present a concise outline of our work.

First of all, we focus on combinatorial constraints on the existence of certain conic arrangements admitting some naturally chosen quasi--homogeneous singularities. Our first result is a Poincar\'e-type polynomial formula for such arrangements that is inspired by a very recent paper by the third author \cite{Pok}. More precisely, if $\mathcal{C} : f=0$ is an arrangement of $k\geq 2$ conics in the complex projective plane that admits $n_r$ ordinary quasi--homogeneous $r$-tuple points with $r\geq 2$, $t_k$ singularities of type $A_k$ with $k \in \{3,5,7\}$, and $j$ singularities of type $J_{2,0}$, then we can define the Poincar\'e-type polynomial associated with $\mathcal{C}$ as
$$\mathfrak{P}(\mathcal{C}, d_3; t):= 1+ 2k t+\left(\sum_{r\geq2}{(r-1)n_r}+t_3+t_5+t_7+2j+2k-d_3\right)t^2,$$
where $d_{3}$ denotes the degree of the third syzygy of the Jacobian ideal $J_f$ with respect to the degree order. Our Theorem \ref{mmm} provides a Poincar\'e-type polynomial identity decoding the combinatorics of some plus--one generated conic arrangements together with the homological properties of the Jacobian ideal $J_{f}$. We show how to use this tool in practice by presenting instructive examples.

Next, we present a tool that might be called a Hirzebruch-type inequality (see Theorem \ref{hirz}) that provides us a concise formula restricting possible combinatorics of a certain class of conic arrangements, namely if $\mathcal{C} \subset \mathbb{P}^{2}_{\mathbb{C}}$ is an arrangement of $k\geq 4$ smooth conics admitting $n_{r}$ ordinary intersection points with $r \in \{2,3,4\}$, $t_{k}$ singularities of type $A_{k}$ with $k \in \{3,5,7\}$, and $j$ singularities of type $J_{2,0}$, then one has
$$8k + n_{2} + \frac{3}{4}n_{3} \geq \frac{5}{2}t_{3} + 5t_{5} + \frac{29}{4}t_{7} + 6j.$$  

Then we perform a classification of plus--one generated conic arrangements with nodes and tacnodes as singularities, and our main result into this direction is Theorem \ref{42}. During our studies on this problem, we discovered some serious computational obstacles. For this reason, we propose two challenging research problems that interested readers might want to tackle. In the last part of the paper, we study examples of strong Ziegler pairs of conic-line arrangements with only nodes and tacnodes as singularities, and we show the existence of such a pair in degree $8$, see Proposition \ref{zzzz}. This observation is quite remarkable, as this is the simplest Ziegler pair of conics and lines, singularity type-wise, according to our knowledge.

We work over the complex numbers in the projective setting and all symbolic computations are performed using \verb}SINGULAR} \cite{Singular}.

\section{Preliminaries}

We start with algebraic preliminaries. We follow the notation introduced in the recent book by Dimca \cite{Dimca}. Let us denote by $S := \mathbb{C}[x,y,z]$ the coordinate ring of $\mathbb{P}^{2}_{\mathbb{C}}$. Let $C$ be a reduced plane curve defined by a homogeneous polynomial $f$ of degree $d$. We denote by $J_{f}$ the Jacobian ideal associated with $f$, i.e., the ideal $J_{f} = \langle \partial_{x}\, f, \partial_{y} \, f, \partial_{z} \, f \rangle$.
\begin{definition}
Let $p$ be an isolated singularity of a polynomial $f\in \mathbb{C}[x,y]$. Since we can change the local coordinates, assume that $p=(0,0)$.
\begin{itemize}
    
\item The number 
$$\mu_{p}=\dim_\mathbb{C}\left(\mathbb{C}\{x,y\} /\bigg\langle \partial_{x}\, f ,\partial_{y}\, f \bigg\rangle\right)$$
is called \textbf{the Milnor number} of $f$ at $p$.

\item The number
$$\tau_{p}=\dim_\mathbb{C}\left(\mathbb{C}\{x,y\}/\bigg\langle f, \partial_{x}\, f ,\partial_{y} \, f \bigg\rangle \right)$$
is called \textbf{the Tjurina number} of $f$ at $p$.

\end{itemize}
\end{definition}

\textbf{The total Tjurina number} of a given reduced curve $C \subset \mathbb{P}^{2}_{\mathbb{C}}$ is defined as
$${\rm deg}(J_{f}) = \tau(C) = \sum_{p \in {\rm Sing}(C)} \tau_{p}.$$ 

Recall that a singularity is called quasi-homogeneous if and only if there exists a holomorphic change of variables so that the defining equation becomes weighted homogeneous. If $C : f=0$ is a reduced plane curve with only quasi-homogeneous singularities, then one has $\tau_{p}=\mu_{p}$ for all $p \in {\rm Sing}(C)$, and eventually 
$$\tau(C) = \sum_{p \in {\rm Sing}(C)} \mu_{p} = \mu(C),$$
which means that the total Tjurina number is equal to the total Milnor number of $C$. In our considerations, we will need also the following numerical invariant of curves.
\begin{definition}
Let $C \subset \mathbb{P}^{2}_{\mathbb{C}}$ be a reduced curve. \textbf{The Arnold exponent} of $C$ is defined as
$$\alpha_{C} =  \min_{p \in {\rm Sing}(C)} {\rm lct}_{p}(C),$$
where  ${\rm lct}_{p}(C)$ denotes the log--canonical threshold of $p \in {\rm Sing}(C)$.
\end{definition}
\begin{remark}
It is well-known that if $p \in C$ is an ordinary singularity of multiplicity $r$, then $${\rm lct}_{p}(C) = \frac{2}{r}.$$
If now $p \in C$ is a tacnode, i.e., an $A_{3}$ singularity, then ${\rm lct}_{p}(C) = 3/4$. For more information on the Arnold exponent, please consult \cite{DimSer}.
\end{remark}

Next, we will need an important invariant that is defined using the syzygies associated with the Jacobian ideal $J_{f}$.
\begin{definition}
Consider the graded $S$-module of Jacobian syzygies of $f$, namely $$AR(f)=\{(a,b,c)\in S^3 : a\partial_{x} \, f + b \partial_{y} \, f + c \partial_{z} \, f = 0 \}.$$
\textbf{The minimal degree of non-trivial Jacobian relations} for $f$ is defined to be 
$${\rm mdr}(f):=\min_{r\geq 0}\{AR(f)_r\neq 0\}.$$ 
\end{definition}
For one of the main definitions in our paper, we recall \textbf{the Milnor algebra}, which is defined as follows $M(f) := S / J_{f}$.
\begin{definition}
We say that a reduced plane curve $C$ of degree $d$ is an \textbf{$m$-syzygy} curve when the associated Milnor algebra $M(f)$ has the following minimal graded free resolution:
$$0 \rightarrow \bigoplus_{i=1}^{m-2}S(-e_{i}) \rightarrow \bigoplus_{i=1}^{m}S(1-d - d_{i}) \rightarrow S^{3}(1-d)\rightarrow S \rightarrow M(f) \rightarrow 0$$
with $e_{1} \leq e_{2} \leq \ldots \leq e_{m-2}$ and $1\leq d_{1} \leq \ldots \leq d_{m}$. The $m$-tuple $(d_{1}, \ldots , d_{m})$ is called the exponents of $C$. Moreover, in this setting ${\rm mdr}(f) = d_{1}$.
\end{definition}
\begin{definition}
A reduced curve $C$ in $\mathbb{P}^{2}_{\mathbb{C}}$ is called \textbf{plus--one generated} with the exponents $(d_1,d_2, d_3)$ if $C$ is $3$-syzygy such that $d_{1}+d_{2}=d$. 
\end{definition}
In order to study plus--one generated reduced curves we will use the following characterization that comes from \cite{DS20}.
\begin{proposition}[Dimca-Sticlaru]
\label{dimspl}
Let $C: f=0$ be a reduced $3$-syzygy curve of degree $d\geq 3$ with the exponents $(d_{1},d_{2},d_{3})$. Then $C$ is plus--one generated if and only if
$$\tau(C) = (d-1)^{2} - d_{1}(d-d_{1}-1) - (d_{3}-d_{2}+1).$$
\end{proposition}
\noindent
\begin{definition}
The \textbf{defect} $\nu(C)$ of a plus--one generated reduced plane curve $C \subset \mathbb{P}^{2}_{\mathbb{C}}$ is defined as
$$\nu(C) = d_{3}-d_{2}+1.$$
\end{definition}
\begin{definition}
The \textbf{$\delta$-level} of a plus--one generated reduced plane curve $C \subset \mathbb{P}^{2}_{\mathbb{C}}$ is defined as
$$\delta L(C) = d_{3} - d_{2} = \nu(C)-1 \geq 0.$$ 
\end{definition}

\begin{definition}
A plus--one generated plane curve $C$ satisfying $\delta L(C)=1$ is called \textbf{minimal plus--one generated}.
\end{definition}
\begin{remark}
In the situation when $\delta L(C) = 0$ we call $C$ a \textbf{nearly free} curve.
\end{remark}
\begin{remark}
\label{cruci}
It is worth recalling also that,  according to \cite[Corollary 5.2]{DS20}, if $C = \bigcup_{i=1}^{k}C_{i}$ is a reduced plane curve of degree $d\geq 3$ such that all irreducible components $C_{i}$ of $C$ are rational, then $d_{m}\leq d-1$. We will use this fact, for instance, in Examples \ref{ex33} and \ref{ex35} below.
\end{remark}
Now we pass to combinatorial preliminaries since we would like to define two notions of combinatorics that can be attached to some reduced plane curves. In the following definition, we refer to types of singularities that are determined by their local normal forms and we will follow Arnold's classification of local normal forms presented in \cite{arnold}. In the setting of our paper, we allow our conic arrangements to have ordinary quasi--homogeneous singularities, which are locally described by $x^{r} + y^{r}=0$ for $r \geq 2$, tangential intersection points $A_{2t+1}$ with $t \in \{1,2,3\}$, and singularities $J_{2,0}$ which have the following local normal form $x^3 + bx^{2}y^{2}+y^{6} = 0$  with $4b^{3}+27 \neq 0$. Finally, let us recall their Milnor numbers (and, due to their quasi-homogeneity, these are also the Tjurina numbers).
\begin{table}[h!]
    \centering
    \begin{tabular}{c|c|c}
     & Number of &  \\
     Singularity & occurrences & $\mu_{p}$ \\ \hline
    ordinary $r$-fold & $n_r$ & $(r-1)^2$  \\
    $A_{3}$ & $t_3$ & 3  \\
    $A_{5}$ & $t_5$ & 5  \\
    $A_{7}$ & $t_7$ & 7  \\
    $J_{2,0}$ & $j$ & 10
    \end{tabular}
    \caption {Milnor numbers of the admissible singular points.}
    \label{table: ADES}
\end{table}

Now we are going to define the combinatorial type of a given curve, and our definition is in the spirit of \cite{cog}.
\begin{definition}
Let $C \subset \mathbb{P}^{2}_\mathbb{C}$ be a reduced curve. Then the \textbf{combinatorial type} of $C$ is defined as
$$W_{C} = (\textbf{i}, \bar{d}, {\rm Sing}(C), \Sigma, \delta, \iota, \textbf{r}),$$
where 
\begin{itemize}
\item the elements of $\textbf{i}$, called the indices, are in bijection with the irreducible components of $C$,
\item $\bar{d} : \textbf{i} \rightarrow \mathbb{N}$ is the degree map that assigns to each irreducible component of $C$ its degree,
\item ${\rm Sing}(C)$ is the set of all singular points of $C$,
\item $\Sigma$ is the set of topological types of the points in $S$,
\item $\delta : {\rm Sing}(C) \rightarrow \Sigma$ assigns to each singular point its topological type\footnote{Recall that the topological type of a singularity $f : (\mathbb{C}^{2}, 0) \to (\mathbb{C}, 0)$ is described by the
(local) embedding of the variety $f^{-1}(0)$ in a neighbourhood of the singular point $0 \in \mathbb{C}^{2}$.},
\item $\iota$ assigns to each singular point $p \in {\rm Sing}(C)$ the irreducible components passing through this point,
\item $\textbf{r} = (r_{j})_{j}$ is the sequence of non-negative integers, where

$r_{j}$ = the number of irreducible components of $C$ containing exactly $j$ singular points.
\end{itemize}
\end{definition}
\begin{definition}
We say that two reduced plane curves $C_{1}$ and $C_{2}$ have the same combinatorial data if the combinatorial types $W_{C_{1}}$ and $W_{C_{2}}$ are equivalent meaning that $\Sigma_{1} = \Sigma_{2}$, $\textbf{r}_{1} = \textbf{r}_{2}$, and there are two bijections $\phi_{r} : \textbf{i}_{1} \rightarrow \textbf{i}_{2}$, $\phi_{S} : {\rm Sing}(C_{1}) \rightarrow {\rm Sing}(C_{2})$ such that $\iota_{1}$ and $\iota_{2}$ are equal, meaning that for all singular points $p$ one has $\iota_{1}(p) = \iota_{2}(\phi_{S}(p))$, $\bar{d}_{1} = \bar{d}_{2} \circ \phi_{r}$, and finally $\delta_{1} = \delta_{2} \circ \phi_{S}$.
\end{definition}
Finally, we come up with a less rigorous definition of combinatorics.
\begin{definition}
Let $C = \{C_{1}, \ldots, C_{k}\} \subset \mathbb{P}^{2}_{\mathbb{C}}$ be a reduced curve such that each irreducible component $C_{i}$ is smooth. The \textbf{weak combinatorics} of the curve $C = \bigcup_{i=1}^{k}C_{i}$ is a vector of the form 
$$\mathcal{K}_{C} = (k_{1}, \ldots, k_{s}; m_{1}, \ldots, m_{p}),$$ where $k_{i}$ denotes the number of irreducible components of $C$ of degree $i$ and $m_{j}$ denotes the number of singular points of a curve $C$ of a given type $M_{j}$.
\end{definition}

We will use the above discussion to define strong Ziegler pairs of plane curves in the last section of our paper.

\begin{remark}
It is natural to ponder the question of whether there exists a relationship between the combinatorics $\mathcal{K}_C$ and $W_C$ for a given reduced curve $C$. It is easy to see that the combinatorics $\mathcal{K}_C$ is weaker than $W_C$. This implies that while one can read from $W_C$ the vector $\mathcal{K}_C$, the converse is not true.
\end{remark}

\section{Combinatorial polynomials associated with plus--one generated arrangement of conics with some quasi-homogeneous singularities}
We start with a preparatory result that will be used in our construction of a combinatorial polynomial that will be associated with a certain class of conic arrangements.
\begin{proposition}
\label{31}
    Let $\mathcal{C}=\{C_1,\ldots, C_k\}\subset\mathbb{P}^2_{\mathbb{C}}$ be an arrangement of $k\geq2$ smooth conics that admits $n_r$ ordinary quasi--homogeneous $r$-tuple points, $t_k$ singularities of type $A_k$ with $k \in \{3,5,7\}$, and $j$ singularities of type $J_{2,0}$. Assume that $\mathcal{C}$ is plus--one generated with exponents $(d_1, d_2, d_3)$, then
    $$d_1d_2+d_3=\sum_{r\geq2}{(r-1)n_r}+t_3+t_5+t_7+2j+2k.$$
\end{proposition}
\begin{proof}
    Recall that if $\mathcal{C}$ is plus--one generated of degree $d:=2k$ with exponents $(d_1,d_2,d_3)$, then by Proposition \ref{dimspl} the total Tjurina number is given by
    $$\tau(\mathcal{C})=(d-1)^2-d_1(d-d_1-1)-(d_3-d_2+1).$$
    Taking into account the types of singularities that are admissible by our arrangements, we have $$\tau(\mathcal{C})=\sum_{p\in\text{Sing}(\mathcal{C})}\mu_p=\sum_{r\geq2}{(r-1)^2n_r}+3t_3+5t_5+7t_7+10j.$$
    Moreover, using B\'ezout theorem applied to our arrangement of conics we get the identity:
    $$4\cdot\binom{k}{2}=2(k^2-k)=\sum_{r\geq2}{\binom{r}{2}}n_r+2t_3+3t_5+4t_7+6j.$$
    Thus we can rewrite $\tau(\mathcal{C})$, remembering that $d=2k$, as
    \begin{align*}
        \tau(\mathcal{C}) &= 2\cdot 2(k^2-k)+1- d_1(d-d_1-1) - d_3+d_2-1 \\ &= \sum_{r\geq2}(r^2-r)n_r+2(2t_3+3t_5+4t_7+6j)-d_1 (d-d_{1})-d_3+(d_1+d_2).
    \end{align*}
    Recall that for $\mathcal{C}$ is plus--one generated curves one has $2k=d_1+d_2$, which gives us
    $$\tau(\mathcal{C})=\sum_{r\geq2}{(r^2-r)n_r}+4t_3+6t_5+8t_7+12j-d_1d_2-d_3+2k.$$
    Therefore, we arrive at the equation
    $$\sum_{r\geq2}{(r-1)^2n_r}+3t_3+5t_5+7t_7+10j=\sum_{r\geq2}{(r^2-r)n_r}+4t_3+6t_5+8t_7+12j-d_1d_2-d_3+2k,$$
    which gives us finally
    \begin{align*}
    d_1d_2+d_3 &= \sum_{r\geq2}{\bigg(r^2-r-(r-1)^2\bigg)n_r}+t_3+t_5+t_7+2j+2k \\
    &=\sum_{r\geq2}{(r-1)n_r}+t_3+t_5+t_7+2j+2k
    \end{align*}
and this completes the proof.
\end{proof}
Now we would like to focus on the combinatorial polynomial that can be associated with any arrangement of conics having singularities prescribed above. If $\mathcal{C}\subset\mathbb{P}^2_\mathbb{C}$ is a plus--one generated arrangement of $k$ conics with exponents $(d_1, d_2, d_3)$, then
\begin{enumerate}
    \item[1)] $d_1+d_2=2k,$ and
    \item[2)] $d_1d_2 + d_{3}=\sum_{r\geq2}{(r-1)n_r}+t_3+t_5+t_7+2j+2k$, which follows from Proposition \ref{31}.
\end{enumerate}
For an $m$-syzygy conic arrangement $\mathcal{C}$ of degree $d=2k\geq 4$ and exponents $(d_1, d_2, d_3 , \ldots , d_{m})$ with $m\geq 3$, let us define the following polynomial
$$\mathfrak{P}(\mathcal{C}, d_3; t):= 1+ 2k t+\left(\sum_{r\geq2}{(r-1)n_r}+t_3+t_5+t_7+2j+2k-d_3\right)t^2.$$

\begin{theorem}
\label{mmm}
In the setting of Proposition \ref{31}, if $\mathcal{C}$ is plus--one generated conic arrangement with exponents $(d_1, d_2, d_3)$, then the polynomial $\mathfrak{P}$ splits over the rationals, and the following identity holds:
    $$\mathfrak{P}(\mathcal{C}, d_3;t)=(1+d_1t)(1+d_2t).$$
\end{theorem}

\begin{proof}
    Using Proposition \ref{31}, we see that $\mathfrak{P}$ has the presentation
    $$\mathfrak{P}(\mathcal{C}, d_3;t)=1+(d_1+d_2)t+d_1d_2t^2 ,$$
    which automatically gives us
    $$\mathfrak{P}(\mathcal{C}, d_3;t)=(1+d_1t)(1+d_2t)$$
    and this completes the proof.
\end{proof}

\begin{example}
\label{ex33}
Here, we consider the unique arrangement, up to projective equivalence and denoted by $\mathcal{C}$, of four conics in the complex projective plane with $12$ tacnodes and no other singularities  \cite[Proposition 7]{GM}. We can compute that
$$\mathfrak{P}(\mathcal{C},h;t) = 1 + 8t + (20-h)t^2,$$
where $h$ is considered as an integer-valued variable that plays the role of $d_3$. Let us recall that, according to \cite[Theorem 2.1]{DimSer} and Remark \ref{cruci}, for a reduced plane curve $C$ of degree $d$ with only ${\rm ADE}$ singularities one has
$$ \alpha_C \cdot d - 2 \leq d_{1} \leq h \leq d-1,$$
where $\alpha_C$ is the Arnold exponent of $C$. In our setting, $\alpha_{\mathcal{C}} = 3/4$, and hence $h \in \{4,5,6,7\}$.
We can easily check that for $h \in \{6,7\}$ our polynomial $\mathfrak{P}(\mathcal{C},h;t)$ does not have rational roots, for $h=4$ we have $\mathfrak{P}(\mathcal{C},4;t) = (1+4t)(1+4t)$, and for $h=5$ we arrive at $\mathfrak{P}(\mathcal{C},5;t) = (1+3t)(1+5t)$. Note that the last case cannot occur because $d_{1} \geq  4$, leaving only the admissible case of $h = 4$. According to \cite{DJP}, we know that our arrangement $\mathcal{C}$ is nearly free with $(d_{1},d_{2},d_{3})=(4,4,4)$, as predicted by our considerations.
\end{example}

Now we pass to a Hirzebruch-type inequality devoted to our conic arrangements.
\begin{theorem}
\label{hirz}
Let $\mathcal{C} \subset \mathbb{P}^{2}_{\mathbb{C}}$ be an arrangement of $k\geq 4$ smooth conics admitting $n_{i}$ ordinary intersection points with $i \in \{2,3,4\}$, $t_{p}$ singularities of type $A_{p}$ with $p \in \{3,5,7\}$, and $j$ singularities of type $J_{2,0}$. Then one has
\begin{equation}
\label{hirze}
8k + n_{2} + \frac{3}{4}n_{3} \geq \frac{5}{2}t_{3} + 5t_{5} + \frac{29}{4}t_{7} + 6j.
\end{equation}
\end{theorem}
\begin{proof}
We will apply Kobayashi's inequality to our plane curves \cite{Kobayashi}, namely if $C$ is a reduced plane curve of degree $d \geq 7$ admitting only ${\rm ADE}$ and simple elliptic (${\rm SE}$ for short) singularities, where the latter are exactly ordinary quadruple points and $J_{2,0}$ singularities, then the following inequality holds:
\begin{equation}
\label{Kob}
\sum_{p \in {\rm Sing}(C) \cap {\rm ADE}} \bigg( \mu_{p}+1 - \frac{2}{|\Gamma(p)|}\bigg) + \sum_{p \in {\rm Sing}(C) \cap {\rm SE}}(\mu_{p}+1) \leq \frac{5}{6}d^{2}-d,
\end{equation}
where $|\Gamma(p)|$ denotes the order of the group naturally associated with ${\rm ADE}$ singularities -- we recall these numbers below for the completeness of our presentation (cf. \cite{Kobayashi}).

\begin{table}[h!]
    \centering
    \begin{tabular}{c|c}
     Singularity type & $|\Gamma(p)|$  \\ \hline
    ordinary node   & $4$  \\
    ordinary triple point & $16$ \\
    $A_{3}$ & $8$   \\
    $A_{5}$ & $12$   \\
    $A_{7}$ & $16$ 
    \end{tabular}
    \caption {Orders of $\Gamma(p)$.}
    \label{table: gammas}
\end{table}
\noindent
Finally, the following na\"ive count holds:
\begin{equation*}
2(k^{2}-k) = n_{2} + 3n_{3} + 6n_{4} + 2t_{3} + 3t_{5} + 4t_{7} + 6j.
\end{equation*}
We use \eqref{Kob} in our setting. After some simple computations, we get
$$\frac{3}{2}n_{2} + \frac{39}{8}n_{3} + 10n_{4} + \frac{15}{4}t_{3} + \frac{35}{6}t_{5} + \frac{63}{8}t_{7} + 11j \leq \frac{10}{3}k^{2}-2k,$$
and after using the above na\"ive combinatorial count we finally obtain
$$32k + 4n_{2} + 3n_{3} \geq 10t_{3} + 20t_{5} + 29t_{7} + 24j,$$
which completes the proof.
\end{proof}
In the example below we show how to merge the above two technical results in order to decide whether a certain weak-combinatorics can lead to a plus--one generated curve.
\begin{example}
\label{ex35}
Consider the following weak-combinatorics
$$\mathcal{K}_{\mathcal{C}} = (k_{2};t_{3},n_{3}) = (5;17,2).$$
We can compute that
$$\mathfrak{P}(\mathcal{C},h;t) = 1+10t + (31-h)t^2,$$
where $h$ is considered as an integer-valued variable that plays the role of $d_3$. Since $\alpha_{\mathcal{C}} = \frac{2}{3}$ we have $h \in \{5,6,7,8,9\}$. We can easily check that for $h \in \{5,8,9\}$ our polynomial does not have rational roots. Moreover, we can observe that
\begin{itemize}
\item for $h=6$ we have $\mathfrak{P}(\mathcal{C},6;t) = (1+5t)(1+5t)$, and
\item for $h=7$ we have $\mathfrak{P}(\mathcal{C},h;t) = (1+4t)(1+6t)$.
\end{itemize}
Note that the last case cannot occur since $d_1 \geq 5$. This means that if the weak combinatorics $\mathcal{K}_{\mathcal{C}}$ can be geometrically realized, then this realization could potentially be an example of a plus--one generated curve. To verify this, we check whether $\mathcal{K}_{\mathcal{C}}$ satisfies \eqref{hirze}. Notice that
$$40 + \frac{3}{2} = 8k+\frac{3}{4}n_{3} \geq \frac{5}{2}t_{3} = \frac{85}{2},$$
a contradiction. Hence $\mathcal{K}_{\mathcal{C}}$ cannot be realized geometrically as a plus--one generated curve.
\end{example}
\begin{remark}
Obviously our techniques have some limitations once we are close to the boundary cases. For instance, using our methods presented above, we cannot decide whether there exists a geometric realization of the weak-combinatorics $\mathcal{K}_{\mathcal{C}} = (k_{2};t_{3}, n_{2},n_{3}) = (5;17,3,1)$ that can lead us to a plus--one generated arrangement of $5$ conics.
\end{remark}
\section{Classification of plus--one generated arrangements of conics with nodes and tacnodes}
In this section we present our third main result of the paper, namely a partial weak-combinatorial classification of plus--one generated conic arrangements with nodes and tacnodes. Our considerations here are motivated by a recent paper \cite{DJP}, where the authors obtained a complete classification of nearly free arrangements of conics with nodes and tacnodes as singularities. Before that, we start with some warm-up that will shed some light on potential weak-combinatorics of our conic arrangements.

\begin{proposition}
    Let $\mathcal{C}=\{C_1, \ldots,C_k\}\subset\mathbb{P}^2_{\mathbb{C}}$ be an arrangement of $k\geq 2$ smooth conics that admits only $n_2$ nodes and $t_3$ tacnodes. Assume that $\mathcal{C}$ is minimal plus--one generated. Then $n_2=2$.
\end{proposition}
\begin{proof}
    Assume that $\mathcal{C}: f=0$ is minimal plus--one generated with $r=\text{mdr}(f)$, then the following equation holds:
    $$r^2-r(2k-1)+(2k-1)^2=\tau(\mathcal{C})+2=n_2+3t_3+2.$$
     By the combinatorial count for conic arrangements with nodes and tacnodes, we know that
    $$4\cdot \binom{k}{2}=n_2+2t_3.$$
    Hence we obtain
    $$r^2-r(2k-1)+(2k-1)^2=4\cdot \binom{k}{2}+t_3+2.$$
    After simple manipulations, we arrive at
    $$r^2-r(2k-1)+2k^2-2k-t_3-1=0.$$
    The above equation has integer roots if the discriminant
    $\Delta_r=-4k^2+4k+4t_3+5$ is non-negative, so we get $t_3\geq k^2-k-\frac{5}{4}$.
    Coming back to the combinatorial count, we have 
    $$2k^2-2k=4\cdot \binom{k}{2}=n_2+2t_3\geq n_2+2\cdot \left(k^2-k-\frac{5}{4}\right),$$
    and we obtain $n_2\leq \frac{5}{2}$. Moreover, if $\mathcal{C}$ is conic arrangement with only nodes and tacnodes, then 
    $$n_{2} = 2k(k-1)-2t_{3} = 2\cdot(k(k-1)-t_{3})$$
    and hence either $n_{2}=0$ or $n_{2}=2$.
    If $n_{2}=0$, then $t_{3}=k(k-1)$ and this is possible only if $k\leq 4$, see \cite[Section 7]{GM}. Moreover, arrangements of $k\leq 4$ conics with $t_{3}=k(k-1)$ are nearly free by \cite[Proposition 1.6]{DJP} and hence they are not minimal plus--one generated, which completes the proof.
\end{proof}
This result is a nice enumerative criterion that can allow us immediately to decide which plus--one generated conic arrangements with nodes and tacnodes are actually minimal.

We are ready to present our main result of this section.
\begin{theorem}
\label{42}
If $\mathcal{C} \subset \mathbb{P}^{2}_{\mathbb{C}}$ is a plus--one generated arrangement of $k\geq 2$ conics with only $n_{2}$ nodes and $t_{3}$ tacnodes and $d_{3}>d_{2}$, then $k \in \{2,3,4\}$. Furthermore, we can geometrically realize the following weak-combinatorics as plus--one generated conic arrangements with nodes and tacnodes:
$$(k_{2};n_{2},t_{3}) \in \{(2;2,1), (3;2,5), (3;4,4), (4;2,11)\}.$$
\end{theorem}
\begin{proof}
Since $\mathcal{C}$ is a plus--one generated arrangement of $k$ conics with only nodes and tacnodes as singularities we have
$$\alpha_{\mathcal{C}}\cdot 2k - 2 \leq d_{1} \leq 2k/2 = k,$$
where $\alpha_{\mathcal{C}}$ is the Arnold exponent. In our situation $\alpha_{\mathcal{C}} = \min\{\frac{3}{4},1\} = \frac{3}{4}$, and this leads us to 
$$\frac{3}{2}k-2 \leq k,$$
which implies that $k\leq 4$.
Now we are going to provide an enumerative description of the weak-combinatorics of our conic arrangements.

We start with $k=2$. In this situation we have that $d_{1} \in \{1,2\}$, $d_{1}+d_{2}=2k$, and $d_{2} < d_{3} \leq 2k-1 = 3$. Recall that if $d_{1}=1$, then $\mathcal{C}$ is either free or nearly free, hence $d_{1}=2$, which also implies that $d_{2}=2$ and $d_{3}=3$. This means that the defect can only be equal to $\nu(\mathcal{C}) =2$ and this implies that $\tau(\mathcal{C}) = 5$. It is easy to see that the only admissible weak-combinatorics is $(k_{2};n_{2},t_{3}) = (2;2,1)$. By \cite[Theorem 3.1]{MP}, this combinatorics can be geometrically realized and hence we have an example of a plus--one generated arrangement with exponents $(d_{1},d_{2},d_{3}) = (2,2,3)$. Our classification for $k=2$ is completed.

Let us pass to the situation with $k=3$. In this case, we have
$$\lceil3k/2-2\rceil = 3 \leq d_{1} \leq k=3,$$
hence $d_{1}=d_{2} = 3$ and $d_{2} < d_{3} \leq 2k-1 = 5$, so we get $d_{3} \in \{4,5\}$. 
This means that $\nu(\mathcal{C}) \in \{2,3\}$. Let us focus on the case $\nu(\mathcal{C})=2$. This tells us that $17 = \tau(\mathcal{C}) = n_{2} + 3t_{3}$ and by the combinatorial count for conics we have $12 = n_{2} + 2t_{3}$, so we get the weak-combinatorics of the form $(k_{2};n_{2},t_{3}) = (3;2,5)$. It is known, by \cite[Theorem 3.1]{MP}, that this weak-combinatorics is geometrically realizable giving us an example of a plus--one generated arrangement with exponents $(d_{1},d_{2},d_{3})=(3,3,4)$. Let us pass to the situation with $\nu(\mathcal{C}) = 3$. Using the same combinatorial argument as above we can observe that the only admissible weak-combinatorics has the form $(k_{2},n_{2},t_{3}) = (3;4,4)$ and our aim now is to present a geometric realization. Consider the arrangement given by 
$$Q(x,y,z) = (x^2+y^2-z^2)(\ell x^2+y^2-z^2)(x^2+\ell y^2-z^2),$$
where $\ell\in\mathbb{C}\backslash\{0,\pm1\}$ is fixed.
It is easy to see that this arrangement of conics has $n_{2}=4$ and $t_{3}=4$. We can compute the minimal free resolution of the Milnor algebra obtaining that this arrangement is plus--one generated with exponents $(d_{1},d_{2},d_{3})=(3,3,5)$, which completes our classification for $k=3$.

Finally, let us focus on $k=4$. We have
$$\lceil3k/2-2\rceil = 4 \leq d_{1} \leq k=4,$$
hence $d_{1}=d_{2} = 4$ and $d_{2} < d_{3} \leq 2k-1 = 7$, so we get $d_{3} \in \{5,6,7\}$.  Let us focus on the case when $\nu(\mathcal{C})=2$, which means that we have constraints $35=\tau(\mathcal{C})=n_{2}+3t_{3}$ and $24 = n_{2}+2t_{3}$. These two conditions give us the weak-combinatorics of the form $(k_{2};n_{2},t_{3}) = (4;2,11)$. It is known, again by \cite[Theorem 3.1]{MP}, that this weak-combinatorics can be geometrically realized and we get an example of plus--one generated arrangement with exponents $(4,4,5)$. This completes the proof of the second part of our statement.
\end{proof}

 In order to complete our classification of plus--one generated conic arrangements with nodes and tacnodes, we have to deal with two remaining subcases for $k=4$, namely either $\nu(\mathcal{C})=3$ or $\nu(\mathcal{C})=4$. Assume that $\nu(\mathcal{C}) = 3$. In this situation, we have  $34 = \tau(\mathcal{C}) = n_{2}+3t_{3}$ and $24 = n_{2}+2t_{3}$, so the unique admissible weak combinatorics has the form $(k_{2};n_{2},t_{3}) = (4; 4, 10)$. Thanks to an interesting result by Megyesi \cite{GM}, we know that there are exactly three multi-parameter families of arrangements consisting of four conics such that $n_{2}=4$ and $t_{3}=10$. After sampling many concrete realizations, we observed that the resulting arrangements are only $4$-syzygy. We now believe that every arrangement of four conics with ten tacnodes and four nodes is $4$-syzygy. To strengthen and justify our claim, we show that the most symmetric family of $4$ conics with $n_2 = 4$ and $t_3 = 10$ is only $4$-syzygy.
\begin{proposition}
\label{con4}
Consider the following one-parameter family of $4$ conics $\mathcal{C}_{r} \subset \mathbb{P}^{2}_{\mathbb{C}}$ given by
$$Q_{r}(x,y,z) = (x^2+y^2-z^2)(x^2+r^2 y^2 - r^2 z^2)(x^2+y^2-r^2 z^2)(r^2 x^2+y^2-r^2 z^2)$$
with $r \in \mathbb{C} \setminus \{0, \pm 1, \pm \iota\}$ and $\iota^{2}+1=0$. 
Then for every admissible parameter $r$ we have $\mathcal{K}_{\mathcal{C}_{r}}=(4;4,10)$, and the arrangement $\mathcal{C}_{r}$ is $4$-syzygy with exponents $(4,5,5,5)$.
\end{proposition}
\begin{proof}
First, we want to detect the admissible parameters $r$, which are the values of $r$ such that the arrangements $\mathcal{C}_r$ have the same combinatorics (i.e., $n_{2} = 4$ and $t_{3} = 10$). To do so, we can use the Gr\"obner basis methods to compute the Gr\"obner cover of the Jacobian ideal $J_r = \langle \partial_x Q_r, \partial_{y} Q_r, \partial_{z} Q_r \rangle$ following the lines of \cite[pp. 99]{Montes}, and for that purpose we can use the following \verb}SINGULAR} routine.
\begin{verbatim}
option(noloadLib);
LIB "all.lib";
proc gr(){
ring R=(0,r), (x,y,z),dp;
option(noredefine); 
short =0;
poly f = (x2+y2-z2)*(x2+r^2*y2-r^2*z2)*(x2+y2-r^2*z2)*(r^2*x2+y2-r^2*z2);
ideal J=jacob(f);
grobcov(J,"showhom",1);
};
\end{verbatim}
Using the above script we can verify that for each parameter $r \in \{0, \pm 1, \pm \iota\}$ the arrangement $\mathcal{C}_{r}$ degenerates from the expected combinatorics, and this completes the first step of our proof. For the second part, we have to compute the minimal free resolution of the Milnor algebra $M_{r} = S / J_{Q_r}$, and we can perform this computation using the following \verb}SINGULAR} routine.
\begin{verbatim}
ring R = (0,r), (x,y,z), (c,dp);
poly f = (x2+y2-z2)*(x2+r^2*y2-r^2*z2)*(x2+y2-r^2*z2)*(r^2*x2+y2-r^2*z2);
ideal I = jacob(f);
syz(I);
\end{verbatim}

Based on \verb}SINGULAR} computations, we can see that for every admissible $r$ the arrangement $\mathcal{C}_{r}$ is $4$-syzygy with exponents $(4,5,5,5)$ and this completes our proof.
\end{proof}
\begin{remark}
One can observe that conic arrangements described in Proposition \ref{con4} are called, according to \cite[Definition 1.2]{her}, curves of type $2B$, i.e., these are $4$-syzygy curves such that the first two exponents satisfy $d_{1}+d_{2} = {\rm deg}(\mathcal{C}_{r})+1 = 9$. 
\end{remark}
We would like to propose the following difficult classification problem.
\begin{problem}
Is it true that every arrangement of $4$ conics with $n_{2}=4$ and $t_{3}=10$ is $4$-syzygy?
\end{problem}

Finally, let us consider the case with $\nu(\mathcal{C}) = 4$. We have $33=\tau(\mathcal{C})=n_{2}+3t_{3}$ and $24 = n_{2}+2t_{3}$, so the unique admissible weak-combinatorics has the form $(k_{2};n_{2},t_{3}) = (4;6,9)$. Similarly to the situation of arrangements of $k=4$ conics with $n_{2}=4$ and $t_{3}=10$, we made several computational experiments and in all cases we observed that considered arrangements with $\mathcal{K}_{\mathcal{C}} = (k_{2};n_{2},t_{3}) = (4;6,9)$ are $5$-syzygy curves with exponents $(d_{1}, \ldots ,d_{5}) = (5, \ldots, 5)$, so these are, according to \cite[Definition 1.2]{her}, curves of type $3$.

Using the same strategy as in Proposition \ref{con4} we can prove the following.
\begin{proposition}
Consider the following one-parameter family of $4$ conics $\mathcal{T}_{r} \subset \mathbb{P}^{2}_{\mathbb{C}}$ given by
$$Q_{r}(x,y,z) = (x^2+y^2+4rxz)(x^2+y^2-4rxz)(x^2+3y^2-18r^2z^2)(x^2+3y^2-16r^2z^2)$$
with $r \in \mathbb{C} \setminus \{0\}$. Then for every admissible $r$ the arrangement $\mathcal{T}_{r}$ is $5$-syzygy with exponents $(d_{1}, \ldots, d_{5}) = (5, \ldots, 5)$ and $\mathcal{K}_{\mathcal{T}_{r}}=(k_{2};n_{2},t_{3})=(4;6,9)$.
\end{proposition}
Finishing this section, we can propose another classification problem.
\begin{problem}
Is it true that every arrangement of $4$ conics with $n_{2}=6$ and $t_{3}=9$ is $5$-syzygy?
\end{problem}
\section{Strong Ziegler pairs}
In the last part of our paper we would like to focus on the simplest possible constructions of strong Ziegler pairs of plane curves. Let us recall that this notion is strictly motivated by a famous example of Ziegler, which is based on two line arrangements having the same intersection lattices (so the same strong combinatorics), but different ${\rm AR}$ modules. This notion was then generalized to reduced plane curves via the weak combinatorics by Cuntz and the third author \cite{CP}. For a given curve $C : f=0$ in $\mathbb{P}^{2}_{\mathbb{C}}$ we will use the notation ${\rm AR}(C)$ or ${\rm AR}(f)$ interchangeably.
\begin{definition}
Let $C_{1}, C_{2} \subset \mathbb{P}^{2}_{\mathbb{C}}$ be reduced curves. We say that $C_{1}, C_{2}$ form a weak Ziegler pair if $\mathcal{K}_{C_{1}} = \mathcal{K}_{C_{2}}$, but the Milnor algebras associated with these two curves have different minimal free resolutions, which is equivalent to say that the modules ${\rm AR}(C_{1})$ and ${\rm AR}(C_{2})$ are different.
\end{definition}
Now, following the ideas of Cuntz and the third author from \cite{CP}, we define a \textit{strong Ziegler pair of curves}.
\begin{definition}
Let $C_{1}, C_{2} \subset \mathbb{P}^{2}_{\mathbb{C}}$ be reduced curves. We say that $C_{1}, C_{2}$ form a strong Ziegler pair if the combinatorial data $W_{C_{1}}$ and $ W_{C_{2}}$ of these curves are equivalent, but the modules ${\rm AR}(C_{1})$ and ${\rm AR}(C_{2})$ are different.
\end{definition}

Let us now look at arrangements of conics with nodes and tacnodes. We start with a plus--one generated conic arrangement $\mathcal{C}$ such that
$$\mathcal{K}_{\mathcal{C}} = (k_{2};n_{2},t_{3}) = (3;4,4)$$
having the following symmetric defining equation
$$Q(x,y,z) = (x^2 + y^{2} - z^{2})(3x^{2}+y^{2}-3z^{2})(x^{2}+3y^{2}-3z^{2}).$$
Now we are going to add some special lines to the above conic arrangement. Define
$$L_{1} : y-x-2z=0, \quad L_{2} : y+x+2z = 0, \quad L_{3} : y-x+2z=0,$$
and consider two conic-line arrangements
$$\mathcal{C}_{1} = \mathcal{C} \cup L_{1} \cup L_{2} \text{ and } \mathcal{C}_{2} = \mathcal{C} \cup L_{1} \cup L_{3}.$$
 
\begin{proposition}
The combinatorial types $W_{\mathcal{C}_{1}}$ and $W_{\mathcal{C}_{2}}$ of arrangements $\mathcal{C}_{1}$ and $\mathcal{C}_{2}$ are equivalent.
\end{proposition}
\begin{proof}
The arrangements admit only nodes and tacnodes as singularities. A simple inspection, based on the figure below, shows that the combinatorial types of curves $\mathcal{C}_{1}$ and $\mathcal{C}_{2}$ are identical. 
\vskip0.5cm
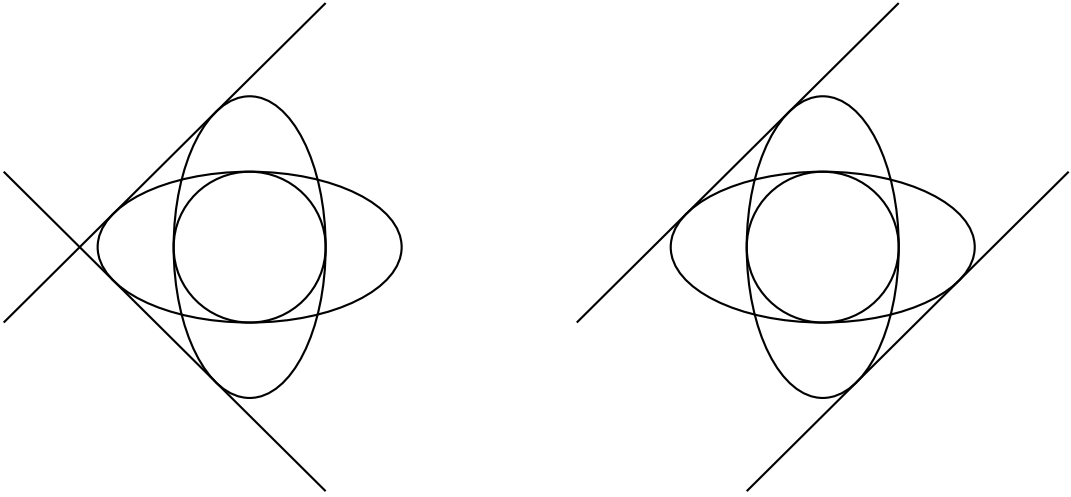
\begin{figure}[h!]
\begin{minipage}{0.45\textwidth}
\centering
    \begin{tikzpicture}[scale=0.6]
\draw[thick] (0,0) circle (1);
\draw[thick] ellipse (2 and 1);
\draw[thick] ellipse (1 and 2);
\draw[thick] (1,-3.2360679774997896964)--(-3.2360679774997896964,1);
\draw[thick] (-3.2360679774997896964,-1)--(1,3.2360679774997896964);
\end{tikzpicture}
\end{minipage} 
\hfill
\begin{minipage}{0.45\textwidth}
\centering
\begin{tikzpicture}[scale=0.6]
\draw[thick] (0,0) circle (1);
\draw[thick] ellipse (2 and 1);
\draw[thick] ellipse (1 and 2);
\draw[thick] (-1,-3.2360679774997896964)--(3.2360679774997896964,1);
\draw[thick] (-3.2360679774997896964,-1)--(1,3.2360679774997896964);
\end{tikzpicture}
\end{minipage}
\label{gr}
\caption{Geometric realizations of arrangements $\mathcal{C}_{1}$ and $\mathcal{C}_{2}$.}
\end{figure}
\end{proof}
As shown in Figure \ref{gr}, the two arrangements differ in principle in the position of one node. In the second case, this node is sent to infinity. While we do not have a rigorous proof, we believe this difference is decisive.
\begin{proposition}
\label{zzzz}
The arrangements $\mathcal{C}_{1}$ and $\mathcal{C}_{2}$ form a strong Ziegler pair.
\end{proposition}
\begin{proof}
Since the curves have the same combinatorial type our proof essentially comes down to showing that the minimal resolutions of the associated Milnor algebras differ. Using \verb}SINGULAR} we can compute the following minimal resolutions:.
\begin{enumerate}
\item[$\mathcal{C}_{1}$] :  $S^{3}(-13) \rightarrow S^{5}(-12) \rightarrow S^{3}(-7) \rightarrow S,$
and
\item[$\mathcal{C}_{2}$] : $S(-14) \oplus S(-13) \rightarrow S(-13) \oplus S^{2}(-12)\oplus S(-11) \rightarrow S^{3}(-7) \rightarrow S,$
\end{enumerate}
which shows that curves $\mathcal{C}_{1}$ and $\mathcal{C}_{2}$ form a strong Ziegler pair.
\end{proof}
It is natural to wonder whether the above example is the smallest possible degree-wise arrangement of conics and lines that form a strong Ziegler pair. We performed computational experiments leading us to the following problem.
\begin{problem}
Is it true that there does not exist any strong Ziegler pairs of $k\leq 2$ conics and $2$ lines with only nodes and tacnodes as singularities?
\end{problem}
Finishing our paper, let us formulate our final question devoted to our Ziegler pair of conics and lines.
\begin{problem}
Can we identify the exact geometric feature that distinguishes our conic-line arrangements in terms of the associated exponents?
\end{problem}
\begin{remark}
After completing the paper, we discovered that our pair of conics and lines appears in a completely different context of Zariski pairs \cite{am}. Recall that Zariski pairs consist of pairs of curves with the same combinatorial type but distinct (embedded) topologies. This connection is incredibly intriguing, and we look forward to exploring this subject further in the near future.
\end{remark}
\section*{Acknowledgments}
We would like to thank an anonymous referee for suggestions, which improved the presentation of our work. We would also like to thank Professors Bannai and Tokunaga for drawing our attention to \cite{am}.

Artur Bromboszcz and Piotr Pokora are supported by the National Science Centre (Poland) Sonata Bis Grant  \textbf{2023/50/E/ST1/00025}. For the purpose of Open Access, the authors have applied a CC-BY public copyright license to any Author Accepted Manuscript (AAM) version arising from this submission.

\bigskip
Affiliation of all authors:
\noindent
Department of Mathematics,
University of the National Education Commission Krakow,
Podchor\c a\.zych 2,
PL-30-084 Krak\'ow, Poland. \\
\nopagebreak
\vskip0.25cm
\noindent
Artur Bromboszcz: \texttt{artur.bromboszcz@uken.krakow.pl}\\
Bartosz Jaros\l awski: \texttt{s168826@student.uken.krakow.pl} \\
Piotr Pokora: \texttt{piotr.pokora@uken.krakow.pl}

\end{document}